\documentclass[10pt]{article}
\usepackage[letterpaper, margin=2.5cm]{geometry}

\usepackage{color, amsmath,amssymb, amsfonts, amstext,amsthm, latexsym}
\usepackage{epsfig}
\usepackage{longtable}
\usepackage{graphicx}
\usepackage{subfigure}

\usepackage{indentfirst}
\usepackage{setspace}

\renewcommand{\phi}{\varphi}

\newtheorem{theorem}{Theorem}
\newtheorem{lemma}{Lemma}
\newtheorem{remark}{Remark}

\newtheorem*{assumption}{Assumption}
\newtheorem{definition}{Definition}

\begin{document}
\date{\today}
\title{A representation formula for the probability density in stochastic dynamical systems with memory}
 \author{\small{\it{Fang Yang$^{a, b}$ and Xu Sun$^{a, b}$\footnote{Corresponding author: xsun@hust.edu.cn}   \footnote{The authors are supported by National Natural Science Foundation Grant 11531006.}} }\\
 \\ \small{$^a$ School of Mathematics and Statistics, Huazhong University of Science and Technology,   Wuhan 430074, Hubei, China.} \\
 \small{$^b$ Center for Mathematical Sciences, Huazhong University of Science and Technology,
  Wuhan 430074, Hubei, China.}
}
\bigskip

\date{Feb. 8th, 2021}
\maketitle

\pagestyle{plain}

Abstract:  Marcus stochastic delay differential equations (SDDEs) are often used to model stochastic dynamical systems with memory in science and engineering.  Since no infinitesimal generators exist for Marcus SDDEs due to the non-Markovian property, conventional Fokker-Planck equations, which govern the evolution behavior of density,  are not available for Marcus SDDEs. In this paper, we identify the Marcus SDDE with some Marcus stochastic differential equation (SDE) without delays but subject to extra constraints. This provides an efficient way to establish existence and uniqueness for the solution, and obtain a representation formula for probability density of the Marcus SDDE. In the formula, the probability density for Marcus SDDE is expressed in terms of that for Marcus SDE without delay.

\bigskip


\section{Introduction}

Stochastic differential equations (SDEs) driven by L\'evy processes are widely used to model stochastic dynamical systems perturbed by non-Gaussian white noises. While It\^o SDEs driven by L\'evy processes are widely applied in biology and finance \cite{Applebaum2009,Duan2015}, Marcus SDEs \cite{Marcus1978,Marcus1981,Kurtz1995} are more appropriate models in physics and engineering \cite{Sun2013,Sun2017}.  To study the propagation and evolution of the uncertainty in stochastic dynamical systems, it is essential to study the dynamics of the probability density, which contains all the statistical information about the uncertainty.  For SDEs without delays, the dynamical behavior of the probability density is governed by Fokker-Planck equations. Fokker-Planck equations corresponding to Marcus SDEs driven by L\'evy processes are derived in \cite{Sun2017}.

Stochastic delay differential equations (SDDEs) are appropriate models for some stochastic dynamical systems with memory, i.e., the future states of the system depend on not only the current but also the past states. The time delays, especially those appearing in diffusion terms, often make it mathematically challenging to study the solution of SDDEs and the associated density. Existence and uniqueness of the solution to  Marcus SDDEs without delays in diffusion terms are established in \cite{K2016}. Generally speaking, Fokker-Planck equations, which contain the adjoint of the infinitesimal generator, are not available for SDDEs due to their non-Markovian property. Recently, the density associated SDDEs driven by Brownian motions is discussed in \cite{Zheng2017G}.

The objective of this paper is twofold: (i) To establish the existence and uniqueness for the solutions to general Marcus SDDEs driven by L\'evy processes. (ii) To derive a representation formula for the probability density associated with these Marcus SDDEs.  This paper is organized as follows. In Section 2,  we present some preliminary results that will be used in later sections. We deal with above objectives (i) and (ii) in Sections 3 and 4, respectively.

\section{Some preliminary results}
First, we introduce some notations for vectors. Throughout this paper, each element in $\mathbb{R}^{d}$ is represented as a column vector. Given $x_1, x_2, \cdots, x_k \in \mathbb{R}^{d}$, then $\left(x_1^{T}, x_2^{T}, \cdots, x_k^{T}\right)^{T}$ is a column vector in $\mathbb{R}^{kd}$. Here $\{\cdot\}^{T}$ represents the transpose of $\{\cdot\}$. 

 Now, we review the definition of Marcus SDEs driven by L\'evy processes and without delays. Consider
\begin{align}\label{e1}
\begin{split}
\begin{cases}
&{\rm d}Z(t)=\alpha(Z(t),t){\rm d}t+\beta(Z(t),t)\diamond {\rm d}L(t), ~~{\rm for}~t>0,\\
&Z(0)=z_0 \in \mathbb{R}^{d},
\end{cases}
\end{split}
\end{align}
where $Z(t)\in \mathbb{R}^{d}$, $\alpha: \mathbb{R}^{d} \times \mathbb{R}^{+}\rightarrow \mathbb{R}^{d}, (x,t)\rightarrow \alpha(x,t)$, $\beta: \mathbb{R}^{d} \times \mathbb{R}^{+} \rightarrow \mathcal{M}^{d\times n}, (x,t)\rightarrow \beta(x,t)=(\beta_{ij})_{d\times n}$ and $L(t)\in \mathbb{R}^{n}$. Here $\mathcal{M}^{d\times n}$ is the set of all $d$-by-$n$ real matrix. 

By L\'evy-It$\hat{\rm o}$ decomposition \cite{Applebaum2009}, the L\'evy process $L(t)$ can be expressed as 
\begin{align}\label{e2}
L(t)=bt+B_{A}(t)+\int_{\|y\|<1}y\widetilde{{N}}(t,{\rm d}y)+\int_{\|y\|\geq 1} yN(t,{\rm d}y),
\end{align}
where $b\in \mathbb{R}^{n}$ is a drift vector, ${B}_{A}(t)$ is the $n$-dimensional Brownian motion with the covariance matrix $A$, $N(t,{\rm d}y)$ is the Poisson random measure defined as 
\begin{align}\label{e3}
N(t,S)(\omega)=\#\left\{ s|0\leq s\leq t; \Delta L(s) (\omega) \in S \right\},
\end{align}
with $\#\{ \cdot \}$ representing the number of elements in the set $``\cdot"$. $S$ is the Borel set in $\mathcal{B}\left(\mathbb{R}^{d}\setminus \{ 0\}\right)$, $\Delta L(s)$ is the jump of $L(s)$ at time $s$ defined as $\Delta L(s)=L(s)-L(s-)$, and $\widetilde{{N}}({\rm d}t,{\rm d}y)$ is the compensated Poisson measure defined as $\widetilde{{N}}({\rm d}t,{\rm d}y)={{N}}({\rm d}t,{\rm d}y)-{\rm d}t\nu({\rm d}y)$.

Note that it's convenient to write the $n$-dimensional Brownian motion $B_A(t)$ in the form of $B_A(t)=\sigma B(t)$ \cite{Applebaum2009}, where $B(t)$ is a standard $n$-dimensional Brownian motion and $\sigma$ is an $n\times n$ nonzero matrix for which $A=\sigma \sigma^{T}$, and $B_A^i(t)=\sum_{j=1}^{n}\sigma_{ij}B_j(t)$ for $i=1,2,\cdots, n$. 

To proceed,  the following definition is needed.

\begin{definition}
	For $u=(u_1,u_2,\cdots, u_{d})^{T} \in \mathbb{R}^{d} $ and $v=(v_1,v_2,\cdots, v_{n})^{T} \in \mathbb{R}^{n}$, the mapping $H$ is defined by 
	\begin{align}\label{e4}
	H: \mathbb{R}^{d}\times \mathbb{R}^{n} \rightarrow \mathbb{R}^{d},~~~(u,v) \mapsto H(u,v)=\Psi(1),
	\end{align}	
	where $\Psi: \mathbb{R} \rightarrow \mathbb{R}^{d},~r\mapsto \Psi(r)$ is the solution of the following ordinary differential equation ${\rm (}$ODE ${\rm )}$
	\begin{align}\label{e5}
	\dfrac{{\rm d}\Psi(r)}{{\rm d}r}= \beta \left(\Psi(r)\right)v, ~~\Psi(0)=u.
	\end{align}
\end{definition}

The $\mathbb{R}^{d}$-valued strong solution $Z(t)$ of Marcus SDE \eqref{e1} is defined  as 
\begin{align}\label{e6}
Z(t)=z_0+\int_{0}^{t}\alpha \left(Z(s), s\right){\rm d}s+\int_{0}^{t}\beta \left(Z(s-),s\right) \diamond {\rm d}L(s),
\end{align}
where the left limit $Z(s-)=\underset{u<s, u\rightarrow s}{{\rm lim} }Z(u)$ and
$``\diamond"$ indicates Marcus canonical integral defined by following.
For each $t\geq 0$,
	\begin{align}\label{e7}
	\begin{split}
	\int_{0}^{t}{\beta }\left(Z(s-),s\right) \diamond {\rm d}L(s)&=\int_{0}^{t}\beta \left(Z(s-),s\right) {\rm d}L(s)+\int_{0}^{t}C(Z(s-),s){\rm d}t\\
	&+\sum_{0\leq s \leq t}\left[H(Z(s-),\Delta L(s))-Z(s-)-\beta (Z(s-),s)\Delta  L(s)\right],
	\end{split}
	\end{align}
	where $C(Z(s-),s)$ is a vector in $\mathbb{R}^{d}$ with the $i$-th element as
	\begin{align}\label{e8}
	C_i(Z(s-),s)=\dfrac{1}{2}\sum_{m=1}^{d}\sum_{j=1}^{n}\sum_{l=1}^{n}\beta_{ml}\left(Z(s-),s\right)\dfrac{\partial}{\partial x_m}\beta_{ij}\left(Z(s-),s\right)A_{lj},\quad~~i=1,2,\cdots, d.
	\end{align}

\begin{remark}
As shown in \cite{Kurtz1995}, the Marcus integral operation $``\diamond"$ satisfies the chain rule of classical calculus, and can be obtained as the limit of Wong-Zakai approximation.
\end{remark}

We have the following result from \cite{Applebaum2009,Kurtz1995,Protter2004}.
	
	\begin{lemma} \label{l1}
	Suppose $\forall x \in \mathbb{R}^{d}$, $\alpha (x,t)$, $\beta(x,t)$ and $\frac{\partial}{\partial x}\beta(x,t)$ satisfy the following two conditions,
\begin{itemize} 
	\item[{\rm (i)}] continuous with respect to $t$;
	\item[{\rm(ii)}] globally  Lipschitz with respect to $x$, i.e., $\forall x_1, x_2 \in \mathbb{R}^{d}$, $t\in \mathbb{R}^{+}$,
	\begin{align}\label{e1_9}
	\begin{split}
	&\|\alpha(x_1,t)-\alpha(x_2,t)\|+\|\beta(x_1,t)-\beta(x_2,t)\|+\bigg\|\frac{\partial }{\partial x}\beta(x_1,t)-\frac{\partial }{\partial x}\beta(x_2,t)\bigg\|\leq C\|x_1-x_2\|,
	\end{split}
	\end{align}
	where $C$ is constant in $\mathbb{R}$.
\end{itemize}
 Then there exists a unique strong solution to the Marcus SDE \eqref{e1}.
	
	\end{lemma}

\section{Existence and uniqueness for Marcus SDDE}
Consider the following Marcus SDDE, 
\begin{align}\label{e9}
\begin{split}
\begin{cases}
&{\rm d}X(t)=f(X(t),X(t-\tau)){\rm d}t+g(X(t), X(t-\tau))\diamond {\rm d}L(t), ~~{\rm for}~t>0,\\
&X(t)=\gamma(t),~~t\in[-\tau, 0],
\end{cases}
\end{split}
\end{align}
where $X(t)$ is a $\mathbb{R}^{d}$-valued stochastic process, $L(t)$ is a $\mathbb{R}^{n}$-valued L\'evy process defined on a probability space $(\Omega, \mathcal{F}, P)$, $\tau \in \mathbb{R}^{+}$ is time delay, $f: \mathbb{R}^{d} \times \mathbb{R}^{d} \rightarrow \mathbb{R}^{d}$, $g=(g_{ij})_{d\times n}: \mathbb{R}^{d}\times \mathbb{R}^{d} \rightarrow \mathcal{M}^{d\times n}$, and $\gamma: [-\tau, 0] \rightarrow \mathbb{R}^d$. 

The ``$\diamond$" in \eqref{e9}, as stated before, can be obtained by Wong-Zakai approximation. However, unlike in \eqref{e1}, there seems no neat expression for ``$\diamond$" in \eqref{e9} due to the delay $\tau$ appearing in the diffusion coefficient $g(X(t), X(t-\tau))$. We circumvent this problem by  associating \eqref{e9} with the following  Marcus SDE without delays,
		\begin{align}\label{e10}
\begin{cases}
{\rm d}X_1(t')=f\left(X_1(t'), \gamma(t'-\tau)\right) {\rm d}t'+g\left(X_1(t'), \gamma(t'-\tau) \right) \diamond {\rm d}L_1(t'),\\
{\rm d}X_2(t')=f\left(X_2(t'), X_1(t')\right) {\rm d}t'+g\left(X_2(t'), X_1(t') \right) \diamond {\rm d}L_2(t'),~~~~~~~~~~~~~~~{\rm for} ~t'\in (0, \tau],\\
~~\vdots~~~~~~~~~~~~~~~~~~~~~~~	\vdots~~~~~~~~~~~~~~~~~~~~~~~~~~~~	\vdots\\
{\rm d}X_k(t')=f\left(X_{k}(t'), X_{k-1}(t')\right) {\rm d}t'+g\left(X_{k}(t'), X_{k-1}(t') \right) \diamond {\rm d}L_k(t'),
\end{cases}
\end{align}
constrained by the condition that the initial value of $X_1(t')$ is prescribed, and  the finial value of $X_{i}(t')$ is set to be equal to  the initial value of $X_{i+1}(t')$ for $i=1,2,\cdots, k-1$, i.e.,
	\begin{equation}\label{e11}
X_1(0)=\gamma_0\quad \text{and} \quad X_{i}(\tau)=X_{i+1}(0).
\end{equation}

In \eqref{e10} and \eqref{e11}, $X_i(t') \in \mathbb{R}^{d} ~(i=1,2, \cdots, k)$ and $\gamma_{0}=\gamma(0)$ is a constant in $\mathbb{R}^d$. Note that condition \eqref{e11} is different from the conventional initial condition 
\begin{align}\label{e12}
	\left(X_1^{T}(0), X_2^T(0), \cdots, X_k^{T}(0)\right)^{T}=\boldsymbol{x}_0,
\end{align}
where $\boldsymbol{x}_0 \in \mathbb{R}^{kd}$ is a constant. Equation \eqref{e10} under condition \eqref{e12}, which is component-wise, can also be rewritten in form of  \eqref{e1}, but with higher dimensionality.

Equation  \eqref{e10} under the condition  \eqref{e11} can be converted from \eqref{e9} in the way given below.  For each solution to Marcus SDDE \eqref{e9}, we can obtain a solution to \eqref{e10} and \eqref{e11} by constructing $X_i(t')=X(t'+(i-1)\tau)$, $L_i(t')=L(t'+(i-1)\tau)-L((i-1)\tau)$ for $i=1,2,\cdots,k$. It is straightforward to check that the path of $X_i(t')$ with $t'\in(0, \tau]$ coincides with the path of $X(t)$ with $t'\in((i-1)\tau, i\tau]$.  Therefore, Marcus SDDE \eqref{e9} is equivalent to Marcus SDE \eqref{e10} under condition \eqref{e11} in the following sense,
\begin{align}\label{e13}
X(t)\overset{a.s.}=
\begin{cases}
X_1(t), \quad & {\rm for} \quad t\in (0,\tau],\\ 
X_2(t-\tau), \quad & {\rm for} \quad t\in (\tau, 2\tau],\\
~~\vdots ~~~~~~~~~~~~~~~~~~~~~~&~\vdots\\
X_{k}(t-(k-1)\tau), \quad & {\rm for} \quad t\in ((k-1)\tau,k\tau],
\end{cases}
\end{align}
or equivalently
\begin{align}\label{e14}
X_{i}(t')\overset{\text{a.s.}}{=}X(t'+(i-1)\tau) \quad {\rm for} \quad t'\in(0,\tau]  \quad {\rm and}  \quad i=1,2,\cdots,k.
\end{align}

The solution to Marcus SDDE \eqref{e9} can be interpreted by Marcus SDE \eqref{e10} under condition \eqref{e11}. We are now ready to establish the existence and uniqueness for the solution to \eqref{e9}. 

We first  introduce some notations that will be used later. Denote $\Phi_k(\gamma_0,t')$, with $\Phi_k: \mathbb{R}^d\times [0, \tau] \rightarrow \mathbb{R}^{kd}, $
$\left(\gamma_{0}, t'\right)\rightarrow \Phi_k(\gamma_{0},t')=\left(X_1^T(t'), X_2^T(t'),\cdots, X_k^T(t')\right)^T$, as the solution at time $t'$ to \eqref{e10} under condition \eqref{e11}. Examining the recursive structure of \eqref{e10} and \eqref{e11}, it is straightforward to check that for $k\geq 2$, $\Phi_{k-1}(\gamma_{0},t')$ coincides with the first $k-1$ components of $\Phi_{k}(\gamma_{0},t')$.

\begin{assumption}[H1]
	Suppose that $\gamma(t)$ is continuous on $[-\tau, 0]$, $f(x,y), g(x,y), \frac{\partial }{\partial x}g(x,y)$ and $\frac{\partial}{\partial y} g(x,y)$ are globally Lipschitz, i.e., $\forall x_1, x_2, y_1, y_2 \in \mathbb{R}^{d}$
	\begin{align}
		\begin{split}
		&\| f(x_1, y_1)-f(x_2,y_2)\|+ \| g(x_1, y_1)-g(x_2,y_2)\|+\bigg\|\frac{\partial }{\partial x}g(x_1,y_1)-\frac{\partial }{\partial x}g(x_2,y_2)\bigg\|+\bigg\|\frac{\partial }{\partial y}g(x_1,y_1)-\frac{\partial }{\partial y}g(x_2,y_2)\bigg\| \\
		&\leq C_1\| x_1-x_2\|+C_2\|y_1-y_2\|.
		\end{split}
	\end{align} 

\end{assumption}

\begin{theorem}
	Under Assumption {\rm (}H1{\rm )}, there exists a unique strong solution to Marcus SDDE \eqref{e9}.
\end{theorem}

\begin{proof}
According to \eqref{e13} or \eqref{e14}, we only need to show that for any given $k\in \mathbb{N}$, there is a unique strong solution to \eqref{e10} under condition \eqref{e11}.  We shall finish the proof by induction.

First, we agrue that the conclusion is true for $k=1$.	In fact, for $k=1$, Marcus SDE \eqref{e10} and \eqref{e11} becomes
	\begin{align}\label{e15}
	\begin{cases}
		{\rm d}X_1(t')={F_1\left(X_1(t'), t'\right)} {\rm d}t'+{G_1\left(X_1(t'), t' \right)} \diamond {\rm d}L_1(t'),~~~t'\in (0,\tau],\\
		X_1(0)=\gamma_0,
		\end{cases}
	\end{align}
	where ${F_1\left(X_1(t'), t'\right)}=f(X_1(t'), \gamma(t'-\tau))$, $G_1\left(X_1(t'), t'\right)=g(X_1(t'), \gamma(t'-\tau))$. It is straightforward to check that {$F_1$, $G_1$} satisfy the conditions in Lemma 1, therefore, \eqref{e15} has a unique strong solution by Lemma 1.

Suppose that the conclusion is true for $k=i$, i.e.,  Marcus SDE
		\begin{align}\label{e16}
\begin{cases}
{\rm d}X_1(t')=f\left(X_1(t'), \gamma(t'-\tau)\right) {\rm d}t'+g\left(X_1(t'), \gamma(t'-\tau) \right) \diamond {\rm d}L_1(t'),\\
{\rm d}X_2(t')=f\left(X_2(t'), X_1(t')\right) {\rm d}t'+g\left(X_2(t'), X_1(t') \right) \diamond {\rm d}L_2(t'),\\
~~\vdots~~~~~~~~~~~~~~~~~~~	\vdots~~~~~~~~~~~~~~~~~~~~~~~~~~~	\vdots\\
{\rm d}X_i(t')=f\left(X_{i}(t'), X_{i-1}(t')\right) {\rm d}t'+g\left(X_{i}(t'), X_{i-1}(t') \right) \diamond {\rm d}L_i(t'),\\
~~X_1(0)=\gamma_0,\,X_2(0)=X_1(\tau),\cdots, \,X_{i}(0)=X_{i-1}(\tau),
\end{cases}
%X_1(0)=x_0
\end{align}
has a unique strong solution
$\left(X_1^T(t'), X_2^{T}(t'), \cdots, X_i^{T}(t')\right)^T.$
 
  It remains to show that the conclusion is true for $k=i+1$, i.e., there is a unique strong solution to Marcus SDE 
\begin{align}\label{e17}
\begin{cases}
{\rm d}X_1(t')=f\left(X_1(t'), \gamma(t'-\tau)\right) {\rm d}t'+g\left(X_1(t'), \gamma(t'-\tau) \right) \diamond {\rm d}L_1(t'),\\
{\rm d}X_2(t')=f\left(X_2(t'), X_1(t')\right) {\rm d}t'+g\left(X_2(t'), X_1(t') \right) \diamond {\rm d}L_2(t'),\\
~~\vdots~~~~~~~~~~~~~~~~~~~	\vdots~~~~~~~~~~~~~~~~~~~~~~~~~~~	\vdots\\
{\rm d}X_i(t')=f\left(X_{i}(t'), X_{i-1}(t')\right) {\rm d}t'+g\left(X_{i}(t'), X_{i-1}(t') \right) \diamond {\rm d}L_i(t'),\\
{\rm d}X_{i+1}(t')=f\left(X_{i+1}(t'), X_{i}(t')\right) {\rm d}t'+g\left(X_{i+1}(t'), X_{i}(t') \right) \diamond {\rm d}L_{i+1}(t'),\\
X_1(0)=\gamma_0,\, X_2(0)=X_1(\tau),\cdots, \,X_{i}(0)=X_{i-1}(\tau), \,X_{i+1}(0)=X_{i}(\tau).
\end{cases}
%X_1(0)=x_0
\end{align}

Note that the existence and uniqueness  for \eqref{e16} implies that  the solution $\left(X_1^T(t'), X_2^T(t'),\cdots,X_i^T(t')\right)^T$ at $t'=\tau$ is a fixed value equal to $\Phi_i(\gamma_0, \tau)$, i.e.,
\begin{align}\label{eq2_22}
	\begin{split}
\left(X_1^T(\tau), X_2^T(\tau), \cdots, X_i^{T}(\tau)\right)^{T}=\Phi_i(\gamma_0, \tau).
	\end{split}
\end{align}
  With \eqref{eq2_22}, equation \eqref{e17} can be written as 
\begin{align}\label{e18}
\begin{cases}
{\rm d}X_1(t')=f\left(X_1(t'), \gamma(t'-\tau)\right) {\rm d}t'+g\left(X_1(t'), \gamma(t'-\tau) \right) \diamond {\rm d}L_1(t'),\\
{\rm d}X_2(t')=f\left(X_2(t'), X_1(t')\right) {\rm d}t'+g\left(X_2(t'), X_1(t') \right) \diamond {\rm d}L_2(t'),\\
~~\vdots~~~~~~~~~~~~~~~~~~~	\vdots~~~~~~~~~~~~~~~~~~~~~~~~~~~	\vdots\\
{\rm d}X_i(t')=f\left(X_{i}(t'), X_{i-1}(t')\right) {\rm d}t'+g\left(X_{i}(t'), X_{i-1}(t') \right) \diamond {\rm d}L_i(t'),\\
{\rm d}X_{i+1}(t')=f\left(X_{i+1}(t'), X_{i}(t')\right) {\rm d}t'+g\left(X_{i+1}(t'), X_{i}(t') \right) \diamond {\rm d}L_{i+1}(t'),
\end{cases}
\end{align}
with initial condition
\begin{align}\label{e19}
\begin{split}
\left(X_1^{T}(0), X_2^T(0),\cdots, X_{i}^T(0), X_{i+1}^T(0)\right)^{T}=\left(\gamma_{0}^{T}, \Phi_{i}^{T}(\gamma_{0},\tau)\right)^{T}.
\end{split}
\end{align}

Equations \eqref{e18} and \eqref{e19} can be rewritten in form of \eqref{e1} as 
\begin{align}\label{e1_21}
\begin{split}
\begin{cases}
	\boldsymbol{X}(t') ={F_{i+1}\left(\boldsymbol{X}(t'),t'\right)}{\rm d}t'+{G_{i+1}\left(\boldsymbol{X}(t'),t'\right)}\diamond {\rm d}\boldsymbol{L}(t'),\\
	\boldsymbol{X}(0)=\boldsymbol{x}_0,
\end{cases}
\end{split}
\end{align}
where 
\begin{align*}
\small
	\begin{split}
		\boldsymbol{X}(t')=\left(                 %左括号
		\begin{array}{ccc}   %该矩阵一共3列，每一列都居中放置
		X_1(t') \\  %第一行元素
		X_2(t')\\  %第二行元素
		\vdots\\
		X_{i}(t')\\
		X_{i+1}(t') 
		\end{array}
		\right),~{F_{i+1}(\boldsymbol{X}(t'),t')}=\left(                 %左括号
		\begin{array}{ccc}   %该矩阵一共3列，每一列都居中放置
		f(X_1(t'), \gamma(t'-\tau)) \\  %第一行元素
		f(X_{2}(t'), X_{1}(t')\\  %第二行元素
		\vdots\\
		f(X_{i}(t'), X_{i-1}(t')\\
		f(X_{i+1}(t'), X_{i}(t') 
		\end{array}
		\right),~\boldsymbol{L}(t')=\left(                 %左括号
		\begin{array}{ccc}   %该矩阵一共3列，每一列都居中放置
		L_1(t') \\  %第一行元素
		L_2(t')\\  %第二行元素
		\vdots\\
		L_{i}(t')\\
		L_{i+1}(t') 
		\end{array}
		\right),~\boldsymbol{x}_0=\left(                 %左括号
		\begin{array}{ccc}   %该矩阵一共3列，每一列都居中放置
		\gamma_{0} \\  %第一行元素
		\Phi_i(\gamma_{0},\tau)
		\end{array}
		\right),
	\end{split}
\end{align*}
and 
\begin{align*}
	\begin{split}
	{G_{i+1}(\boldsymbol{X}(t'),t')}=\left(                 %左括号		
	\begin{array}{ccccc}   %该矩阵一共3列，每一列都居中放置		
	g(X_1(t'), \gamma(t'-\tau)) &  && &\\  %第一行元素		
	& g(X_2(t'),X_1(t')) & &&\\  %第二行元
	&&\ddots	&&\\
	& &&g(X_{i}(t'),X_{i-1}(t')) &\\
	& &&&g(X_{i+1}(t'),X_{i}(t')) \\
	\end{array}	
	\right).
	\end{split}
\end{align*}
It is straightforward to show that {$F_{i+1}, G_{i+1}$} satisfy the conditions in Lemma 1, therefore, \eqref{e1_21} has a unique strong solution by Lemma 1. 
\end{proof}

\section{Representation formula for the density of Marcus SDDE}
\begin{definition}\label{Df2}
	For $k \in \mathbb{N}$, define $\mathcal{Q}_k${\rm :} $\mathbb{R}^{kd}\times [0, \tau]\times \mathbb{R}^{kd}\times [0,\tau]\rightarrow[0, \infty)$, $(\boldsymbol{u},t',\boldsymbol{v},s)\rightarrow \mathcal{Q}_{k}(\boldsymbol{u};t'|\boldsymbol{v};s)$ such that $\forall \boldsymbol{u}, \boldsymbol{v} \in \mathbb{R}^{kd}$ and $0\leq s<t'\leq \tau$, $\mathcal{Q}_k(\boldsymbol{u};t'|\boldsymbol{v};s)$ represents the probability density for the solution $\left( X_1^{T}(t'),X_2^T(t'), \cdots, X_k^T(t')\right)^T$ to Marcus SDE \eqref{e10} at $\boldsymbol{u}$ under the condition $\left( X_1^{T}(s),X_2^T(s), \cdots, X_k^T(s)\right)^T=\boldsymbol{v}$.
\end{definition}

The following three notations are used in this paper to represent probability densities.
\begin{itemize} 
	\item[{\rm (i)}] $\mathcal{P}_{\mathcal{A}}$: the density for the solution $X(t)$ defined in Marcus SDDE \eqref{e9}. For example, $\mathcal{P}_{\mathcal{A}}(x,t)$ represents the density for $X(t)$ at $X(t)=x$. $\mathcal{P}_{\mathcal{A}}(x,3\tau| y, \tau; z, 2\tau)$ represents the conditional density for $X(3\tau)$ at $X(3\tau)=x$ given $X(\tau)=y$ and $X(2\tau)=z$.
	\item[{\rm(ii)}] $\mathcal{Q}_k$: as given in Definition \ref{Df2}, $\mathcal{Q}_k$ is the transitional density of the $\mathbb{R}^{kd}$-valued solution $\left(X_1^T(t'),X_2^T(t'),\cdots,\right.$\\
	$\left.X_k^T(t')\right)^T$ to Marcus SDE \eqref{e10}. For example, $\mathcal{Q}_2(x,y;t'|w,z;s)$ with $0\leq s <t' \leq \tau$ represents the density of $\left(X_1^T(t'),X_2^T(t') \right)^T$ at $X_1(t')=x$ and $X_2(t')=y$ given $X_1(s)=w$ and $X_2(s)=z$.
	\item[{\rm(iii)}] $p$: the density in general cases. For example, $p(X=x;Y=y)$ represents the density of $(X,Y)$ at $X=x$ and $Y=y$; $p(X=x;Y=y|Z=z)$ represents the density of $(X,Y)$ at $X=x$ and $Y=y$ given $Z=z$. Note that $\mathcal{P}_{\mathcal{A}}$ and $\mathcal{Q}_{k}$ can be expressed in terms of $p$. For instance, 
	\begin{align}\label{e20}
	\mathcal{P}_{\mathcal{A}}(x,t)&=p\left(X(t)=x|X(0 \right)=\gamma_{0}),\\
	\mathcal{P}_{\mathcal{A}}(x,3\tau| y, \tau; z, 2\tau)&=p(X(3\tau)=x|X(0)=\gamma_{0}; X(\tau)=y;X(2\tau)=z),\\
	\mathcal{Q}_2(x,y;t'|u,v;s)&=p(X_1(t')=x;X_2(t')=y|X_1(0)=\gamma_0; X_1(s)=u;X_2(s)=v).
	\end{align}
\end{itemize}

To proceed, we need the following Assumptions. 

\begin{assumption}[H2]
	Suppose $\forall \boldsymbol{x}_0 \in \mathbb{R}^{kd}$ with $k \in \mathbb{N}$, the Marcus SDE defined by {\rm (\ref{e10})} under initial condition \eqref{e12} has a unique strong solution.
\end{assumption} 

\begin{assumption}[H3]
	Suppose $\forall \boldsymbol{u, v} \in \mathbb{R}^{kd}$ with $k\in \mathbb{N}$ and $0\leq s<t' \leq \tau$, the probability density  $\mathcal{Q}_{k}(\boldsymbol{u},t'|\boldsymbol{v},s)$, which represents the density of $\boldsymbol{X}(t')$ at $\boldsymbol{u}$ given $\boldsymbol{X}(s)=\boldsymbol{v}$, exists and is strictly positive.
\end{assumption}

Lemma 1 gives a sufficient condition for Assumption (H2) to be true. As for Assumption (H3), there exist sufficient conditions for the existence and regularity of the probability density for some SDEs driven by L\'evy processes with certain restrictions imposed on the jumping measures, see \cite{Applebaum2009,Bichteler1987} and the references therein, among others. However, sufficient conditions for SDEs driven by general L\'evy processes are still not available. There are also some sufficient conditions available for the density to be strictly positive in cases with Gaussian white noise. The strictly positive property of the density for a general class of SDEs can be concluded from the heat kernel estimations \cite{Aronson1968,Davies1989}. A more general sufficient condition for strictive positiveness of density is presented in \cite{Bogachev2009}.

Now, we use the transition density $\mathcal{Q}_{k}$ of Marcus SDE \eqref{e10} to represent the density $\mathcal{P}_{\mathcal{A}}$ of Marcus SDDE \eqref{e9}. 

\begin{theorem}\label{Th2}{\rm \textbf{[Representation formula for the density of Marcus SDDE]}} Suppose that Assumptions {\rm (}H2{\rm )} and  {\rm (}H3{\rm )}  hold. Then $\forall t >0$, the probability density function $\mathcal{P}_{\mathcal{A}}(x,t)$ for the solution $X(t)$ defined by Marcus SDDE {\rm (\ref{e9})} exists. Moreover, $\forall x\in \mathbb{R}^{d}$, the following statements are true.
	\begin{itemize} 
		\item[{\rm (i)}] $For$ $ t\in (0,\tau]$,
		\begin{equation}\label{e23}
		\mathcal{P}_{\mathcal{A}}(x,t)=\mathcal{Q}_1(x;t \big|\gamma_{0};0).
		\end{equation}
		\item[{\rm(ii)}] $For$ $t\in ((k-1)\tau, k\tau)$ $with$ $k\geq 2$ $and$ $k\in {\mathbb{N}}$,
		\begin{equation}\label{e24}
		\begin{split}
		\mathcal{P}_{\mathcal{A}}(x,t)&=\int_{\mathbb{R}^{2 (k-1) d}} \mathcal{Q}_{k-1}\left( x_1, \cdots,x_{k-1};\tau \big|y_1,\cdots,y_{k-1}; t-(k-1)\tau\right)\\
		&\times \mathcal{Q}_{k}\left( y_1, \cdots,y_{k-1}, x;t-(k-1)\tau \big|\gamma_0,x_1,\cdots,x_{k-1}; 0\right)\prod_{i=1}^{k-1}{\rm d}x_i\prod_{i=1}^{k-1}{\rm d}y_i.
		\end{split}
		\end{equation}
\item[{\rm(iii)}] 
$For$ $ t =k\tau$ with $k\geq 2$ and $k\in \mathbb{N}$,
	\begin{equation}\label{e25}
	\mathcal{P}_{\mathcal{A}}(x,t)=\int_{\mathbb{R}^{(k-1) d}}\mathcal{Q}_k(x_1,\cdots, x_{k-1},x;\tau\big|\gamma_{0}, x_1,\cdots,x_{k-1};0)\prod_{i=1}^{k-1}{\rm d}x_i.
	\end{equation}
	\end{itemize}
	
\end{theorem}

\begin{remark}
	Theorem 2 shows that the density for Marcus SDDE  can be expressed in terms of that for Marcus SDEs without delays.  Governing equations for density of Marcus SDEs without delays  have been established,  see \cite{Sun2017}.
\end{remark}

\begin{remark}
	Note that \eqref{e25} can be absorbed into  \eqref{e24}. In fact, $\forall \boldsymbol{u, v} \in \mathbb{R}^{kd}$, $	\mathcal{Q}_{k}(\boldsymbol{u},s|\boldsymbol{v},s)=\underset{t'\rightarrow s}{{\rm lim}} \mathcal{Q}_{k}(\boldsymbol{u},t'|\boldsymbol{v},s)=\delta(\boldsymbol{u}-\boldsymbol{v})$ and  $f(\boldsymbol{u})=\int_{\mathbb{R}^{kd}}\delta(\boldsymbol{u}-\boldsymbol{v})f(\boldsymbol{v}){\rm d}\boldsymbol{v}.$ Therefore, for $t=k\tau$, equation \eqref{e24} becomes to  \eqref{e25}.
	
\end{remark}

\begin{proof}
First, we prove the statement (i) of Theorem \ref{Th2}. In fact, by equation \eqref{e13} or \eqref{e14}, the density of $X(t)$ for $0<t\leq \tau$ defined by Marcus SDDE (\ref{e9}) is the  same as the density of {$X_1(t)$} defined by Marcus SDE (\ref{e10}) with initial value $\gamma_0 \in \mathbb{R}^d$ and $k=1$. Therefore, (i) is true.

Next, we prove the statement (ii) of Theorem \ref{Th2}.~For $t\in ((k-1)\tau, k\tau)$ with $k\geq 2$ and $k\in \mathbb{N}$, as shown  in Appendix,  if Assumption (H3) holds (i.e., $\mathcal{Q}_k$ exists),  then $\mathcal{P}_{\mathcal{A}}(x, t |$ $x_1,\tau;\cdots;x_{k-1},(k-1)\tau;x_k,k\tau)$, $\mathcal{P}_{\mathcal{A}}(x_{k},k\tau|x_1,\tau;\cdots; x_{k-1},(k-1)\tau)$  and $\mathcal{P}_{\mathcal{A}}(x_1,\tau;\cdots;x_{k-1},(k-1)\tau)$ exist, and can be respectively expressed as
\begin{align}\label{e26}
\begin{split}
&\mathcal{P}_{\mathcal{A}}(x,  t |x_1,\tau;\cdots;x_{k-1},(k-1)\tau;x_k,k\tau)\\
=&\int_{\mathbb{R}^{(k-1)d}}\mathcal{Q}_{k}(x_1,\cdots,x_{k-1},x_{k};\tau|y_1,\cdots,y_{k-1},x;t-(k-1)\tau)\\
 &~~~~~~~~\times\dfrac{\mathcal{Q}_{k}(y_1,\cdots,y_{k-1},x;t-(k-1)\tau|\gamma_{0},x_1,\cdots,x_{k-1}; 0) }{\mathcal{Q}_{k}(x_1,\cdots,x_{k};\tau|\gamma_{0},x_1,\cdots,x_{k-1};0)}\prod_{i=1}^{k-1}{\rm d}y_i,
\end{split}
\end{align}
\begin{align}\label{e27}
\mathcal{P}_{\mathcal{A}}(x_{k},k\tau|x_1,\tau; \cdots; x_{k-1},(k-1)\tau)
	=\dfrac{\mathcal{Q}_{k}(x_1,\cdots,x_k; \tau|\gamma_{0},x_1,\cdots,x_{k-1};0)}{\mathcal{Q}_{k-1}(x_1,\cdots,x_{k-1};\tau|\gamma_{0},x_1,\cdots,x_{k-2};0)},
\end{align}
and
\begin{align}\label{e28}
 \mathcal{P}_{\mathcal{A}}(x_1,\tau;\cdots;x_{k-1},(k-1)\tau) 
=\mathcal{Q}_{k-1}\left(x_1,\cdots,x_{k-1};\tau \big|\gamma_0,x_1,\cdots,x_{k-2};0\right).
\end{align}
Further more, the condition density $\mathcal{P}_{\mathcal{A}}(x,t|x_1,\tau;\cdots;x_{k-1},(k-1)\tau)$ exists due to the identity
\begin{align}\label{e1_29}
\begin{split}
&\mathcal{P}_{\mathcal{A}}(x,t|x_1,\tau;\cdots;x_{k-1},(k-1)\tau)\\
=&\int_{\mathbb{R}^{d}}\mathcal{P}_{\mathcal{A}}(x,t|x_1,\tau;\cdots;x_{k-1},(k-1)\tau;x_{k}, k\tau)\times \mathcal{P}_{\mathcal{A}}(x_{k},k\tau|x_1,\tau;\cdots;x_{k-1},(k-1)\tau){\rm d}x_{k}.
\end{split}
\end{align}
Substitute  equations \eqref{e26} and \eqref{e27} into \eqref{e1_29}, we get
\begin{align}\label{e30}
\begin{split}
&\mathcal{P}_{\mathcal{A}}(x,t|x_1,\tau;\cdots;x_{k-1},(k-1)\tau)\\
=&\int_{\mathbb{R}^{(k-1)d}}\int_{\mathbb{R}^{d}}\mathcal{Q}_{k}(x_1,\cdots,x_{k-1},x_{k};\tau|y_1,\cdots,y_{k-1},x;t-(k-1)\tau){\rm d}x_k \\
&~~~~~~~ \times\dfrac{\mathcal{Q}_{k}(y_1,\cdots,y_{k-1},x;t-(k-1)\tau|\gamma_{0},x_1,\cdots,x_{k-1}; 0) }{\mathcal{Q}_{k-1}(x_1,\cdots,x_{k-1};\tau|\gamma_{0},x_1,\cdots,x_{k-2};0)} \prod_{i=1}^{k-1}{\rm d}y_i\\
=&\int_{\mathbb{R}^{(k-1)d}}\mathcal{Q}_{k-1}(x_1,\cdots,x_{k-1};\tau|y_1,\cdots,y_{k-1},x;t-(k-1)\tau)\\
&~~~~~~~\times \dfrac{\mathcal{Q}_{k}(y_1,\cdots,y_{k-1},x;t-(k-1)\tau|\gamma_{0},x_1,\cdots,x_{k-1}; 0) }{\mathcal{Q}_{k-1}(x_1,\cdots,x_{k-1};\tau|\gamma_{0},x_1,\cdots,x_{k-2};0)} \prod_{i=1}^{k-1}{\rm d}y_i\\
=&\int_{\mathbb{R}^{(k-1)d}}\mathcal{Q}_{k-1}(x_1,\cdots,x_{k-1};\tau|y_1,\cdots,y_{k-1};t-(k-1)\tau)\\
&~~~~~~~\times \dfrac{\mathcal{Q}_{k}(y_1,\cdots,y_{k-1},x;t-(k-1)\tau|\gamma_{0},x_1,\cdots,x_{k-1}; 0) }{\mathcal{Q}_{k-1}(x_1,\cdots,x_{k-1};\tau|\gamma_{0},x_1,\cdots,x_{k-2};0)} \prod_{i=1}^{k-1}{\rm d}y_i.
\end{split}
\end{align}
To derive last identity in \eqref{e30}, we use
\begin{equation*}
	\mathcal{Q}_{k-1}(x_1,\cdots,x_{k-1};\tau|y_1,\cdots,y_{k-1},x;t-(k-1)\tau)=\mathcal{Q}_{k-1}(x_1,\cdots,x_{k-1};\tau|y_1,\cdots,y_{k-1};t-(k-1)\tau),	
\end{equation*}
 which follows from the fact that $\left(X_1^T(\tau), X_2^T(\tau),\cdots, X_{k-1}^T(\tau)\right)^T$ in \eqref{e10} only depends on 
 \begin{align*}
 \left(X_1^T(t-(k-1)\tau), X_2^T(t-(k-1)\tau), \cdots, X_{k-1}^T(t-(k-1)\tau)\right)^T
 \end{align*}
  and independent of $X_{k}(t-(k-1)\tau)$. 

By using equations \eqref{e30} and \eqref{e28},  the density $\mathcal{P}_{\mathcal{A}}(x,t)$ of Marcus SDDE \eqref{e9} exists by the identity
\begin{equation} \label{e32}
\begin{split}
\mathcal{P}_{\mathcal{A}}\left(x,t\right)=&\int_{\mathbb{R}^{(k-1)d}} \mathcal{P}_{\mathcal{A}} \left(x,t|x_1,\tau;\cdots;x_{k-1},(k-1)\tau\right)\times \mathcal{P}_{\mathcal{A}}\left(x_1,\tau;\cdots;x_{k-1},(k-1)\tau\right) \prod_{i=1}^{k-1}{\rm d}x_i\\
=&\int_{\mathbb{R}^{2(k-1) d}} \mathcal{Q}_{k-1}\left( x_1, \cdots,x_{k-1};\tau \big|y_1,\cdots,y_{k-1}; t-(k-1)\tau\right)\\
&~~~~~~~~~\times \mathcal{Q}_{k}\left( y_1, \cdots,y_{k-1}, x;t-(k-1)\tau \big|\gamma_0,x_1,\cdots,x_{k-1}; 0\right)\prod_{i=1}^{k-1}{\rm d}x_i\prod_{i=1}^{k-1}{\rm d}y_i.
\end{split}
\end{equation}
Therefore, (ii) is true.

Finally, we show the statement (iii) of Theorem 2 is true. In fact, for $t=k\tau$ with $k\geq 2$ and $k \in \mathbb{N}$,
\begin{align}\label{e33}
\begin{split}
&\mathcal{P}_{\mathcal{A}}(x,k\tau)=\int_{\mathbb{R}^{d}}\mathcal{P}_{\mathcal{A}}(x,k\tau|x_1,\tau;\cdots;x_{k-1},(k-1)\tau)\times \mathcal{P}_{\mathcal{A}}(x_1,\tau;\cdots;x_{k-1},(k-1)\tau)\prod_{i=1}^{k-1}{\rm d}x_{i}.
\end{split}
\end{align}

Substitute equations \eqref{e27} and \eqref{e28} into \eqref{e33}, we get equation \eqref{e25}. Therefor, (iii) is true.

\end{proof}

\section*{Appendix}
\subsection*{{A.1 ~~Proof of equation \eqref{e26}}.} 
{For $t \in ((k-1)\tau, k\tau]$ with $k\geq 2$ and $k\in \mathbb{N}$},
\begin{align}\label{e34}
\begin{split}
&\mathcal{P}_{\mathcal{A}}\left(x,  t|x_1,\tau;\cdots;x_{k-1},(k-1)\tau;x_k,k\tau\right)\\
=&\int_{\mathbb{R}^{(k-1)d}}\mathcal{P}_{\mathcal{A}}\left(y_1, t-(k-1)\tau;y_2, t-(k-2)\tau; \cdots; y_{k-1}, t-\tau; x,t|x_1,\tau;\cdots;x_{k-1}, (k-1)\tau;x_{k},k\tau\right) \prod_{i=1}^{k-1}{\rm d}y_i\\
=&\int_{\mathbb{R}^{(k-1)d}}p(X(t-(k-1)\tau)=y_1;X(t-(k-2)\tau)=y_2;\cdots;X(t-\tau)=y_{k-1}; X(t)=x|\\
&~~~~~~~~~~~~~~X(0)=\gamma_{0};X(\tau)=x_1;\cdots;X((k-1)\tau)=x_{k-1};X(k\tau)=x_k)\prod_{i=1}^{k-1}{\rm d}y_i\\
=&\int_{\mathbb{R}^{(k-1)d}}p(X_1(t-(k-1)\tau)=y_1;\cdots;X_{k-1}(t-(k-1)\tau)=y_{k-1}; X_{k}(t-(k-1)\tau)=x|\\
&~~~~~~~~~~~X_1(0)=\gamma_{0};\cdots; X_{k-1}((k-1)\tau)=X_{k}(0)=x_{k-1},X_{k}(\tau)=x_{k})\prod_{i=1}^{k-1}{\rm d}y_i\\
=&\int_{\mathbb{R}^{(k-1)d}}p(X_1(t-(k-1)\tau)=y_1;\cdots;X_{k-1}(t-(k-1)\tau)=y_{k-1}; X_{k}(t-(k-1)\tau)=x|\\
&~~~~~~~~~~~X_1(0)=\gamma_{0};\cdots;X_{k}(0)=x_{k-1};X_{1}(\tau)=x_1; \cdots;X_{k}(\tau)=x_{k})\prod_{i=1}^{k-1}{\rm d}y_i\\
=&\int_{\mathbb{R}^{(k-1)d}}p(X_1(\tau)=x_1;\cdots; X_{k}(\tau)=x_k|X_1(t-(k-1)\tau)=y_1;\cdots; X_{k}(t-(k-1)\tau)=x)\\
&~~~~~~~\times \dfrac{p(X_1(t-(k-1)\tau)=y_1;\cdots;X_{k}(t-(k-1)\tau )=x|X_1(0)=\gamma_{0};\cdots;X_{k}(0)=x_{k-1}) }{ p(X_1(\tau)=x_1;\cdots; X_{k}(\tau)=x_k|X_1(0)=\gamma_{0};\cdots;X_{k}(0)=x_{k-1})}\prod_{i=1}^{k-1}{\rm d}y_i\\
=&\int_{\mathbb{R}^{(k-1)d}}\mathcal{Q}_{k}(x_1,\cdots,x_{k};\tau|y_1,\cdots,y_{k-1},x;t-(k-1)\tau)\\
&~~~~~~~\times \dfrac{\mathcal{Q}_{k}(y_1,\cdots,y_{k-1},x;t-(k-1)\tau|\gamma_{0},x_1,\cdots,x_{k-1}; 0) }{\mathcal{Q}_{k}(x_1,\cdots,x_{k};\tau|\gamma_{0},x_1,\cdots,x_{k-1};0)}\prod_{i=1}^{k-1}{\rm d}y_i,
\end{split}
\end{align}
To derive the last identity, we use the notation as expressed in (26).  

\subsection*{{A.2 ~~Proof of equation \eqref{e27}}.}

{For $t=k\tau$ with $k\geq 2$ and $k\in \mathbb{N}$},
\begin{align}\label{e35}
\begin{split}
&\mathcal{P}_{\mathcal{A}}(x_{k},k\tau|x_1,\tau;\cdots; x_{k-1},(k-1)\tau)\\
=&p(X(k\tau)=x_{k}|X(0)=\gamma_{0};X(\tau)=x_{1};\cdots;X((k-2)\tau)=x_{k-2}; X((k-1)\tau)=x_{k-1})\\
=&p(X_{k}(\tau)=x_{k}|X_1(0)=\gamma_{0};X_{1}(\tau)=X_{2}(0)=x_1; \cdots;X_{k-1}(\tau)=X_{k}(0)=x_{k-1})\\
=&p(X_{k}(\tau)=x_k|X_1(0)=\gamma_{0};X_{2}(0)=x_1;\cdots;X_{k}(0)=x_{k-1}; X_{1}(\tau)=x_1;\cdots;X_{k-1}(\tau)=x_{k-1})\\
=&\dfrac{p(X_{1}(\tau)=x_1;\cdots;X_{k}(\tau)=x_k|X_1(0)=\gamma_{0};\cdots;X_{k}(0)=x_{k-1})}{p(X_{1}(\tau)=x_1;\cdots;X_{k-1}(\tau)=x_{k-1}|X_1(0)=\gamma_{0};\cdots;X_{k}(0)=x_{k-1})}\\  \vspace{10cm}
=&\frac{p(X_{1}(\tau)=x_1;\cdots;X_{k}(\tau)=x_k|X_1(0)=\gamma_{0};\cdots;X_{k}(0)=x_{k-1})}{p(X_{1}(\tau)=x_1;\cdots;X_{k-1}(\tau)=x_{k-1}|X_1(0)=\gamma_{0};\cdots;X_{k-1}(0)=x_{k-2})} \vspace{10cm} \\ 
=&\frac{\mathcal{Q}_{k}(x_1,\cdots,x_k; \tau|\gamma_{0},x_1,\cdots,x_{k-1};0)}{\mathcal{Q}_{k-1}(x_1,\cdots,x_{k-1};\tau|\gamma_{0},x_1,\cdots,x_{k-2};0)},
\end{split}
\end{align}
where the second last $`` ="$ follows from
\begin{align}\label{e36}
\begin{split}
&p(X_{1}(\tau)=x_1;\cdots;X_{k-1}(\tau)=x_{k-1}|X_1(0)=\gamma_{0};\cdots;X_{k-1}(0)=x_{k-2};X_{k}(0)=x_{k-1})\\
=&p(X_{1}(\tau)=x_1;\cdots;X_{k-1}(\tau)=x_{k-1}|X_1(0)=\gamma_{0};\cdots;X_{k-1}(0)=x_{k-2}),
\end{split}
\end{align}
which is the consequence of the fact that $X_1(\tau), X_2(\tau),\cdots, X_{k-1}(\tau)$ in Marcus SED \eqref{e10} only depends on their initial values $X_1(0), X_2(0), \cdots, X_{k-1}(0)$ and independent of $X_{k}(0)$.

\subsection*{A.3 ~~Proof of equation \eqref{e28}.}
{For $t=k\tau$ with $k\geq 2$ and $k\in \mathbb{N}$},
\begin{align}\label{e37}
\begin{split}
&\mathcal{P}_{\mathcal{A}}(x_1,\tau;x_2,2\tau;\cdots;x_{k-1},(k-1)\tau)\\
=&\mathcal{P}_{\mathcal{A}}(x_1,\tau)\times \mathcal{P}_{\mathcal{A}}(x_2,2\tau|x_1,\tau)\times\cdots \times \mathcal{P}_{\mathcal{A}}(x_{k-1}, (k-1)\tau|x_1,\tau;x_2,2\tau;\cdots;x_{k-2},(k-2)\tau)\\
\vspace{2ex}
=&\mathcal{Q}_{1}\left(x_1;\tau \big|\gamma_0;0\right)\times \dfrac{\mathcal{Q}_2\left(x_1,x_2;\tau \big|\gamma_0,x_1;0\right)}{\mathcal{Q}_{1}\left(x_1;\tau \big|\gamma_0;0\right)} \times \cdots \times \dfrac{\mathcal{Q}_{k-1}\left(x_1,\cdots,x_{k-1};\tau \big|\gamma_0,x_1,\cdots,x_{k-2};0\right)}{\mathcal{Q}_{k-2}\left(x_1,\cdots,x_{k-2};\tau \big|\gamma_0,x_1,\cdots,x_{k-3};0\right)}\\
=&\mathcal{Q}_{k-1}\left(x_1,\cdots,x_{k-1};\tau \big|\gamma_0,x_1,\cdots,x_{k-2};0\right).
\end{split}
\end{align}
Equation (31) has been used to derive the second last \```=" in \eqref{e37}.


\begin{thebibliography}{10}

\bibitem{Applebaum2009}
D.~Applebaum.
\newblock {\em {L}\'evy {P}rocesses and {S}tochastic {C}alculus, 2nd Edition}.
\newblock Cambridge University Press, 2009.

\bibitem{Duan2015}
J.~Duan.
\newblock {\em An {I}ntroduction to {S}tochastic {D}ynamics}.
\newblock Cambridge University Press, 2015.

\bibitem{Marcus1978}
S.~I. Marcus.
\newblock Modeling and {A}nalysis of {S}tochastic {D}ifferential {E}quations
  {D}riven by {P}oint {P}rocesses.
\newblock {\em IEEE Transactiona on Information Theory}, IT-24(2):164--172,
  1978.

\bibitem{Marcus1981}
S.~I. Marcus.
\newblock Modeling and {A}pproximation of {S}tochastic {D}ifferential
  {E}quations {D}riven by {S}emimartingales.
\newblock {\em Stochastics}, 4:223--245, 1981.

\bibitem{Kurtz1995}
T.~Kurtz, E.~Pardoux, and P.~Protter.
\newblock {S}tratonovich {S}tochastic {D}ifferential {E}quations {D}riven by
  {G}eneral {S}emimartingales.
\newblock {\em Annales de l'Institut Henri Poincar\'e-Probabilit\'es et
  Statistiques}, 1(32):351--377, 1995.

\bibitem{Sun2013}
X.~Sun, J.~Duan, and X.~Li.
\newblock {A}n {A}lternative {E}xpression for {S}tochastic {D}ynamical
  {S}ystems with {P}arametric {P}oisson {W}hite {N}oise.
\newblock {\em Probabilistic Engineering Mechanics}, 32:1--4, 2013.

\bibitem{Sun2017}
X.~Sun, X.~Li, and Y.~Zheng.
\newblock Governing {E}quations for {P}robability {D}ensities of {M}arcus
  {S}tochastic {D}ifferential {E}quations with {L}\'evy {N}oise.
\newblock {\em Stochastics and Dynamics}, 17(5):1750033, 2017.

\bibitem{K2016}
K.~K\"ummel.
\newblock {O}n the {D}ynamics of {M}arcus {T}ype {S}tochastic {D}ifferential
  {E}quations.
\newblock {\em Doctoral Dissertation,~Friedrich Schiller Universit\"at Jena},
  2016.

\bibitem{Zheng2017G}
Y.~Zheng and X.~Sun.
\newblock Governing {E}quations for {P}robability {D}ensities of {S}tochastic
  {D}ifferential {E}quations with {D}iscrete {T}ime {D}elays.
\newblock {\em Discrete and Continuous Dynamical Systems-Series B},
  22(9):3615--3628, 2017.

\bibitem{Protter2004}
P.~E. Protter.
\newblock {\em {S}tochastic {I}ntegration and {D}ifferential {E}quations, 2nd
  Edition}.
\newblock Springer, 2004.

\bibitem{Bichteler1987}
K.~Bichteler, J.~Gravereaux, and J.~Jacod.
\newblock {\em Malliavin {C}alculus for {P}rocesses with {J}umps}.
\newblock Gordon and Breach, 1987.

\bibitem{Aronson1968}
D.~G. Aronson.
\newblock Non-negative {S}olutions of {L}inear {P}arabolic {E}quations.
\newblock {\em Annali Della Scuola Normale Superiore DI Pisa-Classe di
  Scienze}, 22:607--694, 1968.

\bibitem{Davies1989}
E.~B. Davies.
\newblock {\em Heat {K}ernels and {S}pectral {T}heory}.
\newblock Cambridge University Press, 1989.

\bibitem{Bogachev2009}
V.~I. Bogachev, M.~Roeckner, and S.~V. Shaposhnikov.
\newblock Positive {D}ensities of {T}ransition {P}robabilities of {D}iffusion
  {P}rocesses.
\newblock {\em Therory of Probability and Its Applications}, 53(2):194--215,
  2009.

\end{thebibliography}
\end{document}